\documentclass[10pt, reqno]{amsart}
\usepackage{amsmath, amsthm, amscd, amsfonts, amssymb, graphicx, color}
\usepackage[bookmarksnumbered, colorlinks, plainpages]{hyperref}
\hypersetup{colorlinks=true,linkcolor=red, anchorcolor=green, citecolor=cyan, urlcolor=red, filecolor=magenta, pdftoolbar=true}
\usepackage{mathrsfs}
\newtheorem{theorem}{Theorem}[section]
\newtheorem{lemma}[theorem]{Lemma}

\theoremstyle{definition}
\newtheorem{definition}[theorem]{Definition}
\newtheorem{example}[theorem]{Example}

\theoremstyle{remark}
\newtheorem{remark}[theorem]{Remark}
\numberwithin{equation}{section}

\begin{document}
\setcounter{page}{1}

\title[Dynamical ideals of non-commutative rings]{Dynamical ideals of non-commutative rings}

\author[Nikolaev]
{Igor V. Nikolaev$^1$}

\address{$^{1}$ Department of Mathematics and Computer Science, St.~John's University, 8000 Utopia Parkway,  
New York,  NY 11439, United States.}
\email{\textcolor[rgb]{0.00,0.00,0.84}{igor.v.nikolaev@gmail.com}}

%\dedicatory{In memory of Ola Bratteli}

\subjclass[2010]{Primary 16D25; Secondary 57N13.}

\keywords{cyclic division algebras, 4-manifolds.}

%\date{Received:  August 14, 2015; Revised: yyyyyy; Accepted: zzzzzz.}

\begin{abstract}
A dynamical analog of the prime ideals for simple non-commutative rings is introduced.
We prove a factorization theorem for the dynamical ideals.  
The result is used to classify the surface knots and links  in  
the smooth 4-dimensional manifolds.

\end{abstract}

\maketitle

%**************************************************************************
\section{Introduction}
%***************************************************************************
The concept of an ideal is  fundamental in  commutative algebra. 
Recall that the  prime factorization in the ring of integers $O_K$ of  
a number field $K$  fails to be unique.  To fix the problem,  one needs 
to complete  the set of the prime numbers of $O_K$ by the   
``ideal numbers'' lying in an  abelian extension of $K$
[Kummer 1847]  \cite{Kum1}.
A description  of the ``ideal number'' in terms of the ring $O_K$ leads to 
the notion of an ideal [Dedekind 1863]  \cite{Ded1}.

Formally,  the ideal of a commutative ring $R$ is defined as a  subset $I\subseteq R$, 
such that $I$ is additively closed  and $IR\subseteq I$.  However,  the true power  of the ideals comes from
 their  geometry.  For instance, if $R$ is the coordinate ring of an affine variety $V$,
 then the prime ideals of $R$ make up a topological space homeomorphic to $V$. 
This  link  between algebra and geometry is  critical, e.g.  for  the Weil Conjectures
[Weil 1949] \cite{Wei1}.

Although the notion of an ideal adapts to the non-commutative rings using the
  left, right and two-sided ideals,  such an approach seems devoid 
of meaningful  geometry. 
The drawback is  that  the ``coordinate rings'' in non-commutative 
geometry are usually  simple, i.e. have only trivial ideals  \cite[Section 5.3.1]{N}. 
Moreover, such rings contain the idempotent elements (projections) and therefore the Ore localization 
of a domain fails in general.

\medskip
The aim of our note is an analog of  the  ideals  for simple non-commutative rings $R$
arising in the geometry of elliptic curves \cite[Section 6.5.1]{N} and topology of the 4-dimensional manifolds \cite{Nik2}.  
Our construction is similar to  [Kummer 1847]  \cite{Kum1}.
Roughly, the idea is this.
 Instead of a subset  $I\subseteq R$,  we take  a partition $\mathscr{D}_{\alpha}$ 
of $R$  by the orbits $\{\alpha^{\mathbf{Z}}(x)~|~x\in R\}$ of an outer automorphism 
$\alpha: R\to R$.   Let $\alpha$ be given by the formula:
%****************************************************************************************
\begin{equation}\label{eq1.1}
\alpha(x)=\mathbf{u}x\mathbf{u}^{-1}, \qquad\forall x\in R,  
\end{equation}
%*****************************************************************************************
where $\mathbf{u}$ is an  ``ideal number'' lying outside $R$. 
We want to extend  $R$ by $\mathbf{u}$ so that $\alpha$ becomes an inner automorphism 
given by the formula
(\ref{eq1.1}). Such an extension is known to coincide with the crossed product $R\rtimes_{\alpha}\mathbf{Z}$,
where $\mathbf{u}$ is the generator of $\mathbf{Z}$. 
  The dynamical system  $\mathscr{D}_{\alpha}:=R\rtimes_{\alpha}\mathbf{Z}$ 
will be called  a  {\it dynamical ideal} (or, dynamial for short) of  $R$.

\medskip
The dynamical system $\mathscr{D}_{\alpha}$ is called   minimal,
if the  $\mathscr{D}_{\alpha}$ does not split   into  a union of 
simpler dynamical  sub-systems.
The $\mathscr{D}_{\alpha}$ is minimal if and only if the crossed product 
$R\rtimes_{\alpha}\mathbf{Z}$ is a simple algebra.
The following model  example shows, that the minimal dynamials
are a proper generalization of  the prime ideals to the case of
 simple non-commutative rings. 
%******************************************************************************
\begin{example}\label{ex1.1}
{\bf (\cite[Section 6.5.1]{N})}
Let $\mathscr{E}(K)$ be a non-singular  elliptic curve over the number field $K$
having the coordinate ring $\mathfrak{V}(\mathscr{E})$. 
Let $\mathscr{E}(\mathbf{F}_p):=\mathfrak{V}(\mathscr{E})/ \mathscr{P}$ be the localization of 
 $\mathscr{E}(K)$ at the prime ideal $\mathscr{P}\subset \mathfrak{V}(\mathscr{E})$  over a prime number $p$. 
Denote by $\mathscr{A}_{\theta}$  the non-commutative torus, i.e. a simple  $C^*$-algebra  $\mathscr{A}_{\theta}$ 
generated by the unitary operators $U$ and $V$ satisfying  the relation 
%*******************************************************************************
\begin{equation}\label{eq1.2.1}
VU=e^{2\pi i\theta}UV,  \quad\hbox{where} \quad \theta\in\mathbf{R}-\mathbf{Q}.
\end{equation}
%***************************************************************************
 It is verified directly, that   the substitution 
%***********************************************************************************************
\begin{equation}\label{eq1.3}
\left\{
\begin{array}{ccl}
U'&=& e^{\pi i ac} ~U^a ~V^c\\
V' &=& e^{\pi i bd} ~U^b ~V^d 
\end{array}
\qquad
\hbox{with} 
\quad
\left(
\begin{matrix}
a & b\cr c & d
\end{matrix}
\right)\in M_2(\mathbf{Z})
\right.
\end{equation}
%*****************************************************************************
brings (\ref{eq1.2.1})  to the form:
%***********************************************************************************************
\begin{equation}\label{eq1.4}
V'U'=e^{2\pi i\theta (ad-bc)}U'V'. 
\end{equation}
%*****************************************************************************
If $\theta$ is a quadratic irrationality, then the  $\mathscr{A}_{\theta}$ is said to have real 
multiplication and  is denoted by $\mathscr{A}_{RM}$. In particular, the 
$\mathscr{A}_{RM}$ is a non-commutative coordinate ring of the elliptic curve $\mathscr{E}(K)$.
Let $\alpha: \mathscr{A}_{RM}\to \mathscr{A}_{RM}$ be the shift automorphism of the $\mathscr{A}_{RM}$. 
Consider a minimal dynamial
%***********************************************************************************************
\begin{equation}\label{eq1.1.4}
\mathscr{D}_p:=\mathscr{A}_{RM}\rtimes_{\alpha} p\mathbf{Z}\cong 
\mathscr{A}_{RM}\rtimes_{L_p}\mathbf{Z}, 
\end{equation}
%*****************************************************************************
where  $L_p$  an endomorphism  (\ref{eq1.3}) of the $\mathscr{A}_{RM}$ 
  corresponding to the matrix: 
 %***********************************************************************************************
\begin{equation}\label{eq1.1.5}
\left(
\begin{matrix}
tr (\varepsilon^{\pi (p)}) & p\cr -1 & 0
\end{matrix}
\right)\in M_2(\mathbf{Z}),
%\right.
\end{equation}
%*****************************************************************************
see \cite[Section 6.5.3.2]{N} for the notation.
Let $K_0(\mathscr{D}_p)$ be the $K_0$-group of the 
 $C^*$-algebra $\mathscr{D}_p$. 
Then  for all but a finite set of primes $p$,  there exists a canonical 
 isomorphism of the finite abelian groups:
 %****************************************************************************************
\begin{equation}\label{eq1.2}
K_0(\mathscr{D}_p)\cong\mathfrak{V}(\mathscr{E})/ \mathscr{P}. 
\end{equation}
%*****************************************************************************************
\end{example}
%****************************************************************************

\medskip
%*****************************************************************************************
\begin{remark}\label{rmk1.2}
Letting $p$ in  (\ref{eq1.2})  run through the set of all primes, 
one gets a bijective map between the minimal dynamials  $\mathscr{D}_p$ of  the non-commutative 
 ring $\mathscr{A}_{RM}$  and   the prime ideals $\mathscr{P}$ of  the commutative ring 
 $\mathfrak{V}(\mathscr{E})$.   In fact,  such a map is a functor between the respective
 categories  \cite[Section 6.5.1]{N}.  
 A relation of $K_0(\mathscr{D}_p)$ to the Brauer group is discussed in remark \ref{rmk4.7}. 
 Formula  (\ref{eq1.2}) is a motivation of the following definition. 
 \end{remark}
%***************************************************************************************
%*************************************************************************
\begin{definition}\label{dfn1.3}
By the {\it Dedekind-Hecke} ring $R$ we understand a simple 
non-commutative topological ring,  such that $\mathbf{Z}\subseteq Out~R$,
where $Out~R$ are  the outer automorphisms of $R$. 
\end{definition}
%************************************************************************
%******************************************************************************
\begin{example}\label{ex1.4}
The  $\mathscr{A}_{\theta}$ is a Dedekind-Hecke ring in the norm topology. 
Indeed, one gets  from (\ref{eq1.3}) and (\ref{eq1.4}),  that $Out ~\mathscr{A}_{\theta}\cong SL_2(\mathbf{Z})$.  
The upper triangular matrix $\left(\begin{smallmatrix} 1 & 1\cr 0 & 1  \end{smallmatrix}\right)$ is the generator
of a group $\mathbf{Z}\subset Out ~\mathscr{A}_{\theta}$. 
\end{example}
%****************************************************************************

\medskip
%*****************************************************************************************
\begin{remark}\label{rmk1.5}
If $R$ is a Dedekind-Hecke ring, then 
%*************************************************************************
\begin{equation}\label{eq1.1.7}
R\rtimes_{\alpha} m\mathbf{Z}\cong
R_m\rtimes_{\alpha} \mathbf{Z}:=
R\rtimes_{\alpha_m} \mathbf{Z},
\end{equation}
%***************************************************************************
where $\alpha_m: R\to R_m\subseteq R$ is an endomorphism of degree $m\ge 1$. 
Indeed, the  $R\rtimes_{\alpha} m\mathbf{Z}$  gives rise to an automorphism  of $R$ acting by the 
formula $x\mapsto \mathbf{u}^mx\mathbf{u}^{-m}$. 
Denote by $\hat\alpha_m$ an extension of $\alpha_m$ to the $R\rtimes_{\alpha} \mathbf{Z}$. 
Since  $\mathbf{u}^m\in Ker~\hat\alpha_m$,
one gets an isomorphism $R\rtimes_{\alpha} m\mathbf{Z}\cong
R_m\rtimes_{\alpha} \mathbf{Z}$.  
Notice that  the crossed product $R\rtimes_{\alpha_m} \mathbf{Z}$  is 
undefined, since $\alpha_m$ is not an automorphism of $R$,   if  $m\ne 1$.  
Hence the $R\rtimes_{\alpha_m} \mathbf{Z}$ in (\ref{eq1.1.7}) is  a symbolic notation   
for either $R\rtimes_{\alpha} m\mathbf{Z}$ or 
$R_m\rtimes_{\alpha} \mathbf{Z}$, see  (\ref{eq1.1.4}). 
\end{remark}
%****************************************************************************

%*****************************************************************************************
\begin{remark}\label{rmk1.6}
 The endomorphism  (\ref{eq1.1.5})  of the ring $\mathscr{A}_{RM}$   is known to commute with
the Hecke operator $T_p$ on a lattice $\Lambda\subset\mathbf{C}$, such that  $\mathscr{E}(K)\cong\mathbf{C} / \Lambda$,
see  \cite[Section 6.5.1]{N}.   Hence our terminology in \ref{dfn1.3}. 
\end{remark}
%****************************************************************************

\bigskip
Our main result can be formulated as follows. 
Fix a generator $\alpha: R\to R$ of the cyclic group $\mathbf{Z}\subseteq Out ~R$. 
 For brevity, we write $\mathscr{D}_m :=R\rtimes_{\alpha} m\mathbf{Z}$, where $m\ge 1$ is an integer. 
The product of the dynamials $\mathscr{D}_{m_1}\cong (R, G_1,\pi)$ and  
$\mathscr{D}_{m_2}\cong (R, G_2,\pi)$ will be defined as the direct sum of the corresponding 
transformation groups, i.e. 
%*************************************************************************
\begin{equation}\label{eq1.7}
\mathscr{D}_{m_1}\mathscr{D}_{m_2}:=(R, G_1\oplus G_2,\pi).
\end{equation}
%***************************************************************************

\bigskip
 %*************************************************************************
\begin{theorem}\label{thm1.2}
The dynamials  $\mathscr{D}_m$ of the Dedekind-Hecke ring $R$   satisfy the fundamental theorem  of arithmetic,
i.e. 
%************************************************************************************
\begin{equation}\label{eq1.8}
\mathscr{D}_m=\mathscr{D}_{p_1^{k_1}}\mathscr{D}_{p_2^{k_2}}\dots \mathscr{D}_{p_n^{k_n}},
\end{equation}
%*********************************************************************************
where $\prod_{i=1}^n p_i^{k_i}$ is the prime factorization of $m$.  The product (\ref{eq1.8}) is unique up to 
the order of the factors. In particular, the  dynamial  $\mathscr{D}_m$ is minimal if and only if $m=p$ is
a prime number. 
\end{theorem}
%************************************************************************
The paper is organized as follows. The preliminary results can be found in Section 2.
Theorem \ref{thm1.2} is proved in Section 3. An application of theorem \ref{thm1.2}
to  the classification of the surface knots and  links in the 4-dimensional manifolds  is considered in Section 4.

%**************************************************************************
\section{Preliminaries}
%***************************************************************************
We briefly review the topological dynamical systems,  cyclic division algebras, arithmetic topology of 3-manifolds  and knotted surfaces
in 4-manifolds. 
For a detailed exposition we refer the reader to [Gottschalk \& Hedlund 1955] \cite[Chapter 2]{GH},  [Pierce 1982] \cite[Chapter 15]{P} ,  [Morishita 2012] \cite{M}
and [Piergallini 1995] \cite{Pie1}, respectively.

%**************************************************************************
\subsection{Topological dynamics}
%***************************************************************************
Let $X$ be a topological space  and let $T$ be a topological group. 
Consider a continuous map $\pi: X\times T\to X$, such that:  

\medskip
(i) $\pi(x,e)=x$ for the identity $e\in T$ and all $x\in X$;

\smallskip
(ii) $\pi(\pi(x,t),s)=\pi(x, ts)$ for all $s,t\in T$ and all $x\in X$.  

\medskip\noindent
The triple $(X,T,\pi)$ is called a transformation group (topological dynamical system). 
A topological isomorphism of $(X,T,\pi)$ onto $(Y,S,\rho)$ is a couple $(h,\varphi)$ consisting 
of a homeomorphism $h: X\to Y$ and a homeomorphic group-isomorphism $\varphi: T\to S$, 
such that $(xh, t\varphi)\rho=(x,t)\pi h$ for all $x\in X$ and $t\in T$. The transformation groups
 $(X,T,\pi)$ and $(Y,S,\rho)$ are said to be equivalent, if  $(X,T,\pi)\cong (Y,S,\rho)$ are topologically
 isomorphic. 

Let $S\subseteq T$ be a subgroup and $x\in X$.  The $S$-orbit of $x$ is  
a subset $O_S(x)=\{xS ~|~S\subseteq T\}$ of $X$. If $S=T$ we omit the 
subscript and write  an orbit as  $O(x)$. 
The orbit $O_S(x)$ is the least $S$-invariant subset of $X$.  
The closure of an $S$-orbit in $X$
is denoted by $\bar O_S(x)$.  
The set $A\subseteq X$ is called $S$-minimal provided $\bar O_S(x)=A$
for all $x\in A$ and $A$ does not contain a smaller $S$-orbit closure. 
We call $A$ a minimal set when $S=T$.  

A subset $S\subset T$ is said to be left (right, resp.) syndetic in $T$,
if $T=SK$ ($T=KS$, resp.) for some compact subset $K\subset T$. 
In particular, if $T$ is discrete and if $S$ is a subgroup of $T$, 
then $S$ is syndetic in $T$ if and only if $S$ is a finite index subgroup of $T$
 [Gottschalk \& Hedlund 1955] \cite[Remark 2.03 (7)]{GH}. 
 We shall use the following fact.
 %************************************************************************
 \begin{theorem}\label{thm2.1}
 {\bf \cite[Theorem 2.32]{GH}}
 Let $X$ be compact and minimal under $T$. Let $S$ be a syndetic invariant
 subgroup of $T$. Then the $S$-orbit closures define a (star-closed) decomposition 
 of $X$. 
 \end{theorem}
 %************************************************************************

%**************************************************************************
\subsection{Cyclic division algebras}
%***************************************************************************
The real quaternions have been the single example of a division ring 
(hyper-algebraic number field) until the discovery in 1906 of the cyclic division 
algebras  by Leonard ~E. ~Dickson. Roughly speaking, such algebras are an infinite family of the division 
rings generalizing the quaternions and represented by matrices over
% the cyclic  extension of 
an algebraic number field. Let us review the main ideas.  

Let $K$ be a number field and let $E$ be a finite Galois extension of $K$. 
Denote by $G=Gal ~(E|K)$ the Galois group of $E$ over $K$. 
Let $n=\dim_K (E)$ be the dimension of $E$ as a vector space over $K$. Consider 
the ring $End_K(E)$ of all $K$-linear transformations of $E$. Fixing a basis of $E$
over $K$, one gets an isomorphism:
%***************************************************************************
\begin{equation}\label{eq2.1}
End_K(E)\cong M_n(K). 
\end{equation}
%****************************************************************************
Denote by $\mathscr{C}\subset End_K(E)$ a subring generated by multiplications
by the elements $\alpha\in E$ and the automorphisms $\theta\in G$. 
It can be verified directly, that the commutation relation between the two is given by the 
formula: 
%***************************************************************************
\begin{equation}\label{eq2.2}
\theta\alpha=\theta(\alpha)\theta. 
\end{equation}
%****************************************************************************

\smallskip
Further we restrict to the case when $G\cong \left(\mathbf{Z}/n\mathbf{Z}\right)^{\times}$ is a cyclic group of order $n$ 
generated by $\theta$.  Thus the relation (\ref{eq2.2}) in $\mathscr{C}$ is complemented by the
 relation $\theta^n=1$.  On the other hand, it is easy to see that $\theta$ is an invertible element of $\mathscr{C}$ 
 along with any element of the form $\gamma\theta$, where $\gamma\in E$. Notice that 
%***************************************************************************
\begin{equation}\label{eq2.3}
(\gamma\theta)^n=N(\gamma)\theta^n=N(\gamma), 
\end{equation}
%****************************************************************************
where $N(\gamma)\in K^{\times}$ is the $K$-norm of the algebraic number $\gamma$.   
%*************************************************************************
\begin{definition}\label{dfn2.2}
The cyclic algebra $\mathscr{C}(a)$ is a subring of the ring $M_n(K)$ 
generated by  the elements $\alpha\in E$ and the element $u:=\gamma\theta$ satisfying the relations:
%***************************************************************************
\begin{equation}\label{eq2.4}
u\alpha=\theta(\alpha)u, \quad u^n=a\in K^{\times}. 
\end{equation}
%****************************************************************************
\end{definition}
%************************************************************************
%******************************************************************************
\begin{example}\label{exm2.3}
Let $K\cong \mathbf{R}$ and $E\cong\mathbf{C}$. 
Then $G\cong\left(\mathbf{Z}/2\mathbf{Z}\right)^{\times}$ 
and  $\theta$ is the complex conjugation. In this case 
$\mathscr{C}(1)\cong M_2(\mathbf{R})$ and $\mathscr{C}(-1)\cong\mathbb{H}$,
where $\mathbb{H}$ is the algebra of real quaternions. 
\end{example}
%****************************************************************************

\medskip
The $\mathscr{C}(a)$ is a simple algebra of dimension $n^2$ over $K$. The
field $K$ is the center of $\mathscr{C}(a)$ and $E$ is the maximal subfield
of $\mathscr{C}(a)$.  The following theorem gives the necessary and sufficient 
condition for the $\mathscr{C}(a)$ to be a division algebra. 
%************************************************************************
 \begin{theorem}\label{thm2.4}
 {\bf (Wedderburn's Norm Criterion)}
 The $\mathscr{C}(a)$ is a division algebra if and only if $a^n$ is the least
 power of $a$ which is the norm of an element in $E$. 
  \end{theorem}
 %************************************************************************
%*************************************************************************
\begin{lemma}\label{lm2.5}
The $\mathscr{C}(a)$ is a Dedekind-Hecke ring. 
\end{lemma}
%*****************************************************************************
\begin{proof}
Recall that all automorphisms of the matrix algebra $M_n(K)$ are inner. 
Since $\mathscr{C}(a)\subset M_n(K)$, we conclude that 
%******************************************************************************
\begin{equation}
Out~\mathscr{C}(a)\subset M_n(K)
\end{equation}
%********************************************************************
and the group $Out~\mathscr{C}(a)$ consists of the elements 
$g\in M_n(K)$, such that $g\mathscr{C}(a)g^{-1}=\mathscr{C}(a)$. 
It is clear that $\mathbf{Z}\subset Out ~\mathscr{C}(a)$, if $g$
has the infinite order, e.g. is given by an upper triangular matrix. 
On the other hand, the   $\mathscr{C}(a)$ is a topological ring
endowed e.g. with the discrete topology. Thus  $\mathscr{C}(a)$
is a Dedekind-Hecke ring, see definition \ref{dfn1.3} 
\end{proof}

%**************************************************************************
\subsection{Arithmetic topology}
%***************************************************************************
Such a theory  studies an interplay between
the topology of 3-dimensional manifolds and the algebraic number fields 
 [Morishita 2012] \cite{M}.  
 Let $\mathfrak{M}^3$ be a category of  closed 3-dimensional manifolds,
such that  the arrows of $\mathfrak{M}^3$ are  homeomorphisms  between the  manifolds.
Likewise, let $\mathbf{K}$ be a category of the algebraic number fields,  where 
the arrows of $\mathbf{K}$ are  isomorphisms between such fields.
 Let  $\mathscr{M}^3\in \mathfrak{M}^3$ be a 3-dimensional manifold,  let $S^3\in \mathfrak{M}^3$ be the 3-sphere
 and let $O_K$ be the ring of integers of  $K\in\mathbf{K}$.  An exact relation between the 3-dimensional manifolds and 
 the number fields can be described  as follows. 
 %**********************************************************************
\begin{theorem}\label{thm2.5}
{\bf (\cite[Theorem 1.2]{Nik1})}
The exists a covariant functor $F: \mathfrak{M}^3\to \mathbf{K}$, such that:

\medskip
(i) $F(S^3)=\mathbf{Q}$; 

\smallskip
(ii)  each  ideal $I\subseteq O_K=F(\mathscr{M}^3)$ corresponds to 
a link $\mathscr{L}\subset\mathscr{M}^3$; 

\smallskip
(iii)  each  prime ideal $\mathscr{P}\subseteq O_K=F(\mathscr{M}^3)$ corresponds to
a knot  $\mathscr{K}\subset\mathscr{M}^3$. 
\end{theorem}
%*******************************************************************************************

\medskip
%*****************************************************************************************
\begin{remark}\label{rmk2.6}
The domain of $F$ extends to  the smooth 4-dimensional manifolds $\mathscr{M}^4$ \cite[Theorem 1.1]{Nik2}. 
The range of $F$  consists of the cyclic division algebras $\mathscr{C}(a)$.
Since $\mathscr{C}(a)$ is a simple algebra, we must use the dynamical ideals $\mathscr{D}_m(\mathscr{C}(a))$ 
instead of the ideals. The   $\mathscr{D}_m(\mathscr{C}(a))$ correspond to the surface knots and links 
in $\mathscr{M}^4$. We refer the reader 
 to Section 4 for the details.
 \end{remark}
%****************************************************************************

%**************************************************************************
\subsection{Knotted surfaces in 4-manifolds}
%***************************************************************************
By $\mathscr{M}^4$ we understand a smooth 4-dimensional manifold. 
Let $S^4$ be the 4-dimensional sphere and $X_g$ be a closed
2-dimensional orientable surface of genus  $g\ge 0$.
%*****************************************************************************************
\begin{definition}
By the knotted surface  $\mathscr{X}:=X_{g_1}\cup\dots X_{g_n}$ in  $\mathscr{M}^4$  one understands  a 
transverse immersion of a collection of $n\ge 1$  surfaces  $X_{g_i}$ into $\mathscr{M}^4$, i.e.:
%*******************************************************************************
\begin{equation}
\iota:  ~X_{g_1}\cup\dots X_{g_n}\hookrightarrow \mathscr{M}^4. 
\end{equation}
%*******************************************************************************
We  refer to $\mathscr{X}$ a {\it surface knot} if $n=1$ and a {\it surface link} if $n\ge 2$. 
\end{definition}
%***************************************************************************************

\medskip
The following result extends  the well-known theorem
on  the covering  of the 3-dimensional sphere $S^3$ branched  
over a link in the $S^3$. 
%***************************************************************************************
\begin{theorem}\label{thm2.8}
{\bf ([Piergallini 1995]  \cite{Pie1})} 
Each smooth 4-dimensional  manifold $\mathscr{M}^4$ is the 4-fold PL cover of the sphere $S^4$
 branched at the points of a knotted surface $\mathscr{X}\subset S^4$. 
\end{theorem}
%**********************************************************************************

%**************************************************************************
\section{Proof of theorem \ref{thm1.2}}
%***************************************************************************
We shall split the proof in a series of lemmas. 
%**************************************************************************
\begin{lemma}\label{lm3.1}
$\mathscr{D}_{m_1}\mathscr{D}_{m_2}=\mathscr{D}_{m_2}\mathscr{D}_{m_1}$ 
for any integers $m_1,m_2\ge 1$.
\end{lemma}
%************************************************************************
\begin{proof}
Consider an exact sequence of the subgroups $G_1\cong m_1\mathbf{Z}$ and 
$G_2\cong m_2\mathbf{Z}$ of the (additive) abelian group $(m_1m_2)\mathbf{Z}$:
%*************************************************************************************
\begin{equation}\label{eq3.1}
0\to m_1\mathbf{Z}\to (m_1m_2)\mathbf{Z}\to m_2\mathbf{Z}\to 0.
\end{equation}
%*************************************************************************************

\medskip
It can be verified directly, that the exact sequence (\ref{eq3.1}) splits. We conclude
therefore that: 
%*************************************************************************************
\begin{equation}\label{eq3.2}
 (m_1m_2)\mathbf{Z}\cong G_1\oplus G_2. 
\end{equation}
%*************************************************************************************

\medskip
On the other hand, using formula (\ref{eq1.7}) one gets the following sequence of isomorphisms: 
%********************************************************************************
\begin{eqnarray}\label{eq3.3}
\mathscr{D}_{m_1}\mathscr{D}_{m_2} &\cong&   
(R, (m_1m_2)\mathbf{Z}, \pi)\cong\cr 
&&\cr
&\cong&  (R, (m_2m_1)\mathbf{Z}, \pi)\cong 
\mathscr{D}_{m_2}\mathscr{D}_{m_1}. 
\end{eqnarray}
%*********************************************************************************

\medskip
The conclusion of lemma \ref{lm3.1} follows from formula (\ref{eq3.3}). 
\end{proof}

%**************************************************************************
\begin{lemma}\label{lm3.2}
$\mathscr{D}_{m_1}\mathscr{D}_{m_2}=\mathscr{D}_{m_1m_2}$ 
for any integers $m_1,m_2\ge 1$.
\end{lemma}
%************************************************************************
\begin{proof}
We preserve the notation used in the proof of lemma \ref{lm3.1}. 
From the definition  (\ref{eq1.7})  of a product of the dynamials  and 
formula (\ref{eq3.2}), one gets the following sequence of the isomorphisms:
%********************************************************************************
\begin{eqnarray}\label{eq3.4}
\mathscr{D}_{m_1}\mathscr{D}_{m_2} &\cong&   
(R, G_1\oplus G_2, \pi) \cong\cr 
&&\cr
&\cong& (R, (m_1m_2)\mathbf{Z}, \pi)  \cong 
\mathscr{D}_{m_1m_2}. 
\end{eqnarray}
%*********************************************************************************

\medskip
The conclusion of lemma \ref{lm3.2} follows from formula (\ref{eq3.4}).

\end{proof}

%**************************************************************************
\begin{lemma}\label{lm3.3}
The index map $i(\mathscr{D}_m))=m$ is an isomorphism between the multiplicative 
semigroup $(\mathfrak{D}_R, \times)$ of all dynamials $\mathfrak{D}_R$  of the ring $R$ and the multiplicative 
semigroup $(\mathbf{N}, \times)$ of all positive integers $\mathbf{N}$. 
\end{lemma}
%************************************************************************
\begin{proof}
(i) The unit of the semigroup $(\mathfrak{D}_R, \times)$  is the dynamial $\mathscr{D}_1\cong R\rtimes_{\alpha}\mathbf{Z}$. 
The $i(\mathscr{D}_1)=1$ is the unit of the semigroup  $(\mathbf{N}, \times)$ and $\mathscr{D}_1$ is the unique dynamial having
index $1$.   Therefore, the index map is injective.

\medskip
(ii) It is easy to see, that the index map is surjective. Indeed,  for every $m\ge 1$ there exists a dynamial $\mathscr{D}_m$, 
such that $i(\mathscr{D}_m)=m$.

\medskip
(iii) Let us verify, that the index map preserves the products, i.e.
$i(\mathscr{D}_{m_1}\mathscr{D}_{m_2})=i(\mathscr{D}_{m_1})i(\mathscr{D}_{m_2})$. 
Using lemma \ref{lm3.2},   we  calculate: 
%*********************************************************************************
\begin{equation}\label{eq3.5}
i(\mathscr{D}_{m_1}\mathscr{D}_{m_2})=i(\mathscr{D}_{m_1m_2})=m_1m_2=i(\mathscr{D}_{m_1})i(\mathscr{D}_{m_2}). 
\end{equation}
%******************************************************************************* 

\medskip
Since the index map is (i) injective, (ii) surjective and (iii) preserves the products, we conclude that such a map is an isomorphism
of the underlying semigroups.  Lemma \ref{lm3.3} follows. 
\end{proof}

%**************************************************************************
\begin{lemma}\label{lm3.4}
%************************************************************************************
If $m=p_1^{k_1}p_2^{k_2}\dots p_n^{k_n}$ is the prime factorization of an integer  $m\ge 1$,
then $\mathscr{D}_m=\mathscr{D}_{p_1^{k_1}}\mathscr{D}_{p_2^{k_2}}\dots \mathscr{D}_{p_n^{k_n}}$.
The latter product is unique up to  the order of the factors. 
\end{lemma}
%************************************************************************
\begin{proof}
Let $m\ge 1$ be an integer. According to the fundamental theorem of arithmetic, there exists a prime
factorization: 
%*********************************************************************************
\begin{equation}\label{eq3.6}
m=p_1^{k_1}p_2^{k_2}\dots p_n^{k_n},
 \end{equation}
%******************************************************************************* 
where $p_i$ are the prime numbers and $k_i$ are positive integers. 
Moreover,  the product (\ref{eq3.6})  is unique up to the order of the factors.

\medskip
(i) Using lemma \ref{lm3.2},  we calculate:
%*********************************************************************************
\begin{equation}\label{eq3.7}
\mathscr{D}_m=\mathscr{D}_{p_1^{k_1}p_2^{k_2}\dots p_n^{k_n}}=
\mathscr{D}_{p_1^{k_1}}\mathscr{D}_{p_2^{k_2}}\dots \mathscr{D}_{p_n^{k_n}}. 
 \end{equation}
%******************************************************************************* 
 
\medskip
(ii)  Since by lemma \ref{eq3.3} the index map is an isomorphism,  we conclude
that the product (\ref{eq3.7}) is unique up to the order of the factors. (Otherwise 
such a property would  fail also  for the product (\ref{eq3.6}).)  

\medskip
The conclusion of lemma \ref{lm3.4} follows from items (i) and (ii). 
\end{proof}

%**************************************************************************
\begin{lemma}\label{lm3.5}
%************************************************************************************
 The  dynamial  $\mathscr{D}_m$ is minimal if and only if $m=p$ is
a prime number. 
\end{lemma}
%************************************************************************
\begin{proof}
(i) Let us assume that $m=p$ is a prime number. In view of lemma \ref{lm3.4},
the dynamical system $\mathscr{D}_p$ cannot be a product of two or more
dynamical sub-systems.  In other words, the   $\mathscr{D}_p$ is a minimal dynamial.

\medskip
(ii)  Conversely, let $\mathscr{D}_m$ be a minimal dynamical system. 
Let $(R, m\mathbf{Z}, \pi)$ be the corresponding transformation group. 
Notice  that the abelian group  $m\mathbf{Z}$ is discrete.  Therefore every 
finite index subgroup $S$  of  $m\mathbf{Z}$ must be syndetic, see Section 2.1.  
Let us now assume to the  contrary, that $m\ne p$.  Then there exists a non-trivial 
syndetic subgroup $S$ of  the group $m\mathbf{Z}$.  By Theorem \ref{thm2.1}
there exists a (star-closed) decomposition of the dynamical system $\mathscr{D}_m$ by the
$S$-orbit closures, i.e.  by the smaller dynamical sub-systems.   
In other words, the dynamical system  $\mathscr{D}_m$ is not minimal. 
This contradiction  proves  the necessary condition of lemma \ref{lm3.5}.

\medskip
Lemma \ref{lm3.5} is equivalent to items (i) and (ii). 
\end{proof}

\bigskip
Theorem \ref{thm1.2} follows from lemmas \ref{lm3.4} and \ref{lm3.5}.

%**************************************************************************
\section{Knotted surfaces in 4-manifolds}
%***************************************************************************
We  apply  theorem \ref{thm1.2}  to prove an analog of  theorem \ref{thm2.5} for   the smooth 4-dimensional manifolds $\mathscr{M}^4$,
see remark \ref{rmk2.6}. 
Roughly speaking, we replace the ideals $I\subseteq O_K$  by the dynamical ideals  
$\mathscr{D}_m(\mathscr{C}(a))$, where $\mathscr{C}(a)$ is a cyclic division algebra associated
to $\mathscr{M}^4$ \cite{Nik2}.  It is well known, that any knot or link in $\mathscr{M}^4$ is trivial. 
This fact can be derived from the simplicity of  algebra $\mathscr{C}(a)$ and  the Wedderburn  Theorem,  see  \cite{Nik3}.
On the other hand,  there exist non-trivially knotted surfaces $X_g$ in  $\mathscr{M}^4$ (E.~Artin).  
We show that the dynamials $\mathscr{D}_m(\mathscr{C}(a))$ classify  the embeddings
 $X_g\hookrightarrow \mathscr{M}^4$.  Let us pass to a detailed argument.

\medskip
Let  $\mathfrak{M}^4$ be a category of  smooth  4-dimensional manifolds $\mathscr{M}^4$,
such that  the arrows of $\mathfrak{M}^4$ are  homeomorphisms  between the  manifolds.
Denote by  $\mathfrak{C}$ a category of the cyclic division algebras $\mathscr{C}(a)$, 
such that the arrows of $\mathfrak{C}$ are  isomorphisms between these algebras.
The following result is an extension of  Theorem \ref{thm2.5} to $\mathfrak{M}^4$.
%***************************************************************************
\begin{theorem}\label{thm4.1}
The exists a covariant functor $F: \mathfrak{M}^4\to\mathfrak{C}$,
such that: 

\medskip
(i) $F(S^4)=\mathbb{H}(\mathbf{Q})$, where $\mathbb{H}(\mathbf{Q})$ are the rational quaternions; 

\smallskip
(ii) the dynamial  $\mathscr{D}_{m}(F(\mathscr{M}^4))$ corresponds to 
a surface link $\mathscr{X}\subset\mathscr{M}^4$; 

\smallskip
(iii)   the minimal dynamial $\mathscr{D}_{p}(F(\mathscr{M}^4))$ corresponds to
a surface knot  $\mathscr{X}\subset\mathscr{M}^4$. 
\end{theorem}
%***************************************************************************
\begin{proof}
The proof of existence and a detailed construction of functor $F$ can be found in  
\cite[Theorem 1.1]{Nik2}.   To prove items (i)-(iii) of theorem \ref{thm4.1}, 
we shall use Theorem \ref{thm2.5} and the spinning of a link construction dating back to E.~ Artin, 
see e.g. [Kamada 2002] \cite[Chapter 10]{K}.  Let us prove the following lemmas. 

%**************************************************************************
\begin{lemma}\label{lm4.2}
Let $E$ be a cyclic extension of the  number field $K$ 
with  the Galois group $G\cong \left(\mathbf{Z}/n\mathbf{Z}\right)^{\times}$ of order $n$ 
generated by an element $\theta$, see Section 2.2.  Denote by $O_E$ the ring of integers of $E$. 
Then
%*********************************************************************
\begin{equation}\label{eq4.1}
\mathscr{C}(a)\cong O_E\rtimes_{\theta} G. 
\end{equation}
%**************************************************************************
\end{lemma}
%************************************************************************
\begin{proof}
Recall that the algebra $\mathscr{C}(a)$ is generated by multiplications by the 
elements $\alpha\in E$ and the automorphism $\theta\in G$, see Section 2.2.
Since $E\cong O_E\otimes\mathbf{Q}$, we can always assume that $\alpha\in O_E$. 

On the other hand, the commutation relation (\ref{eq2.2}) can be written as
%*********************************************************************
\begin{equation}
\theta(\alpha)=\theta\alpha\theta^{-1},  \quad \forall\alpha\in O_E.  
\end{equation}
%**************************************************************************
Thus the algebra $\mathscr{C}(a)$ is an extension of the ring $O_E$ by 
an element $\theta$, such that each automorphism of $O_E$ becomes an 
inner automorphism. In other words, the  $\mathscr{C}(a)$ coincided with the 
crossed product $O_E\rtimes_{\theta} G$. Lemma \ref{lm4.2} is proved. 
\end{proof}

%**************************************************************************
\begin{lemma}\label{lm4.3}
The dynamial  $\mathscr{D}_m (\mathscr{C}(a))$ defines an ideal 
$I_E\subseteq O_E$, such that:

\medskip
(i) $|O_E/I_E|=m$;

\smallskip
(ii) $\mathscr{D}_m (\mathscr{C}(a))= \mathscr{D}(\mathscr{C}_m(a))$,
where $\mathscr{C}_m(a)=I_E\rtimes_{\theta} G$. 
\end{lemma}
%************************************************************************
\begin{proof}
From (\ref{eq4.1}) and definition of $\mathscr{D}_m$,  one gets: 
%********************************************************************************
\begin{eqnarray}\label{eq4.4}
\mathscr{D}_m (\mathscr{C}(a)) &\cong&   
\mathscr{C}(a)\rtimes m\mathbf{Z}  \cong\cr 
&&\cr
&\cong& \left(O_E\rtimes_{\theta} G\right) \rtimes m\mathbf{Z}. 
\end{eqnarray}
%*********************************************************************************

\medskip
On the other hand, by the commutativity of crossed products by the abelian groups,
we have: 
%*********************************************************************
\begin{equation}\label{eq4.5}
 \left(O_E\rtimes_{\theta} G\right) \rtimes m\mathbf{Z}\cong
  \left(O_E  \rtimes m\mathbf{Z}\right)  \rtimes_{\theta} G. 
\end{equation}
%**************************************************************************

\medskip
Finally,  repeating the argument of remark \ref{rmk1.5}, one gets an isomorphism: 
%*********************************************************************
\begin{equation}\label{eq4.6}
O_E  \rtimes m\mathbf{Z}\cong I_E\rtimes\mathbf{Z}, 
\end{equation}
%**************************************************************************
where $I_E\subseteq O_E$ is an ideal, such that $|O_E/I_E|=m$. 

\medskip
Altogether (\ref{eq4.4})-(\ref{eq4.6}) give us: 
 %*********************************************************************
\begin{eqnarray}\label{eq4.7}
\mathscr{D}_m (\mathscr{C}(a)) &\cong&  
  \left( I_E\rtimes\mathbf{Z}\right)  \rtimes_{\theta} G\cong\cr
  &&\cr
 &\cong &    \left( I_E     \rtimes_{\theta} G  \right)   \rtimes\mathbf{Z}\cong
\mathscr{D}(\mathscr{C}_m(a)),     
  \end{eqnarray}
%**************************************************************************
where $\mathscr{C}_m(a):=I_E\rtimes_{\theta} G$. 
Lemma \ref{lm4.3} is proved. 
\end{proof}

\medskip
%*****************************************************************************************
\begin{remark}\label{rmk4.4}
It is not hard to see, that the crossed product
%*************************************************************************************
\begin{equation}
 \{R\rtimes_{\theta} G ~|~R ~\hbox{is a commutative ring}\}
 \end{equation}
 %******************************************************************
corresponds to a ``spinning'' of the underlying topological space around a 2-dimen\-sional plane
 [Kamada 2002] \cite[Chapter 10]{K}. Indeed, if $R\cong O_E$, then the underlying 
 topological space is a 3-dimensional manifold $\mathscr{M}^3$,  such that $O_E=F(\mathscr{M}^3)$, 
 see Theorem \ref{thm2.5}.  By (\ref{eq4.1}) the spinning $O_E\rtimes_{\theta} G\cong \mathscr{C}(a)$ gives
 us a 4-dimensional manifold $\mathscr{M}^4$, such that $\mathscr{C}(a)=F(\mathscr{M}^4)$
 \cite[Theorem 1.1]{Nik2}.  Likewise, if $R\cong I_E\subseteq O_E$ is an ideal, then the underlying 
 topological space is a link $\mathscr{L}\subset\mathscr{M}^3$,  such that $I_E=F(\mathscr{L})$,
 see Theorem \ref{thm2.5}.  In view of lemma \ref{lm4.3} (ii),  the spinning  $I_E\rtimes_{\theta} G\cong \mathscr{C}_m(a)$
 gives us a 2-dimensional knotted surface $\mathscr{X}\subset \mathscr{M}^4$, such that $\mathscr{C}_m(a)=F(\mathscr{X})$. 
\end{remark}
%***************************************************************************************

\bigskip
Returning to the proof of theorem \ref{thm4.1}, we shall proceed in the following steps.

\bigskip
(i)  Consider a cyclic division algebra $\mathscr{C}(a)$ over the number field $K$.
It is not hard to see, that the smallest cyclic division sub-algebra of  $\mathscr{C}(a)$
consists of the rational quaternions $\mathbb{H}(\mathbf{Q})$, see Example \ref{exm2.3}. 
In other words, each $\mathscr{C}(a)$ can be seen as  an extension of the rational quaternions:
%*************************************************************************************
\begin{equation}
\mathbb{H}(\mathbf{Q})\subseteq \mathscr{C}(a). 
 \end{equation}
 %******************************************************************
Since  the $\mathbb{H}(\mathbf{Q})$ is an algebra over $\mathbf{Q}$ and
%*************************************************************************************
\begin{equation}
\mathbf{Q}\subseteq K,
 \end{equation}
 %******************************************************************
we conclude from Theorem \ref{thm2.5} (i) and remark \ref{rmk4.4},
that:
%*************************************************************************************
\begin{equation}
F(S^4)=\mathbb{H}(\mathbf{Q}),
 \end{equation}
 %******************************************************************
where $S^4$ is the 4-dimensional sphere.  Item (i) of theorem \ref{thm4.1} is proved.

\bigskip
(ii) Let $m=p_1^{k_1}p_2^{k_2}\dots p_n^{k_n}$ be the prime factorization of an integer $m\ge 1$. 
By lemma \ref{lm2.5},  the  $\mathscr{C}(a)$ is a Dedekind-Hecke ring.  Thus  we can apply the prime factorization 
theorem \ref{thm1.2} to the dynamial $\mathscr{D}_m(\mathscr{C}(a))$, i.e. 
%*************************************************************************************
\begin{equation}\label{eq4.11}
\mathscr{D}_m(\mathscr{C}(a))=\mathscr{D}_{p_1^{k_1}}(\mathscr{C}(a)) ~\mathscr{D}_{p_2^{k_2}}(\mathscr{C}(a))
~\dots ~\mathscr{D}_{p_n^{k_n}}(\mathscr{C}(a)). 
\end{equation}
 %******************************************************************

\medskip
On the other hand, each $\mathscr{D}_i(\mathscr{C}(a))$ defines an ideal $I_i\subseteq O_E$,
where $E$ is the maximal number field of the cyclic division algebra $\mathscr{C}(a)$, see
lemma \ref{lm4.3}.  Moreover, item (i) of the same lemma implies  that the prime 
factorization (\ref{eq4.11}) defines  the prime factorization of the ideal $I_m\subset O_E$ corresponding to the 
$\mathscr{D}_m(\mathscr{C}(a))$,  i.e.
%*************************************************************************************
\begin{equation}\label{eq4.12}
I_m=\mathscr{P}_1^{k_1} \mathscr{P}_2^{k_2}\dots \mathscr{P}_n^{k_n},
\end{equation}
 %******************************************************************
where $\mathscr{P}_i$  are the prime ideals of $E$. 
We can now use Theorem \ref{thm2.5} and remark \ref{rmk4.4} 
to conclude that the dynamial  $\mathscr{D}_m(\mathscr{C}(a))$
corresponds to a surface link $\mathscr{X}\subset \mathscr{M}^4$ 
obtained by the spinning of a link $\mathscr{L}$ defined by the
prime factorization (\ref{eq4.12}).   Item (ii) of theorem \ref{thm4.1} is proved.

\bigskip
(iii) Let $m=p$ be a prime number. 
We repeat the argument of item (ii) and we apply item (iii) of Theorem \ref{thm2.5}. 
In this case one gets a surface knot  $\mathscr{X}\subset \mathscr{M}^4$,
since there is only one component in the prime factorization formula (\ref{eq4.11}). 
Item (iii) of theorem \ref{thm4.1} is proved. 

\bigskip
Theorem \ref{thm4.1} is proven. 
\end{proof}

\bigskip
%******************************************************************************
\begin{example}\label{ex4.2}
{\bf (Sphere knots in simply connected 4-manifolds)}

\medskip
(i) Let $\mathscr{W}^4$ be a simply connected 4-dimensional manifold. 
Denote by $K^{ab}$ an abelian extension of the field $\mathbf{Q}$ 
and by $\mathscr{C}(a, K^{ab})$ a cyclic division algebra over the $K^{ab}$.
It follows from \cite[Corollary 1.2]{Nik2}, that 
%**************************************************************************
\begin{equation}\label{eq4.13}
F(\mathscr{W}^4)= \mathscr{C}(a, K^{ab}). 
\end{equation}
%**************************************************************************
Indeed, since the Galois group of $\mathscr{C}(a)=F(\mathscr{W}^4)$
is abelian, so does any subfield of the $\mathscr{C}(a)$. In particular,
the center  of $\mathscr{C}(a)$  must be an abelian extension of $\mathbf{Q}$.

\bigskip
(ii) Consider a sphere  knot:  
%**************************************************************************
\begin{equation}\label{eq4.14}
S^2\hookrightarrow \mathscr{W}^4.
\end{equation}
%**************************************************************************
Theorem \ref{thm4.1} (iii) says that (\ref{eq4.14})  are classified 
by the minimal dynamials:
%**************************************************************************
\begin{equation}\label{eq4.15}
\mathscr{D}_p(\mathscr{C}(a, K^{ab}))=\mathscr{C}(a, K^{ab})\rtimes p\mathbf{Z}\cong \mathscr{C}(a, K_p^{ab}),
\end{equation}
%**************************************************************************
where $K_p^{ab}$ is an abelian extension of the field $\mathbf{Q}_p$ of the $p$-adic numbers. 
The local class field theory  says that the  fields $K_p^{ab}$ are classified by the  open
subgroups of the group  $\mathbf{Q}_p^{\times}$ of index $n=\deg~(K^{ab}|\mathbf{Q})$.  
Thus the classification of the sphere knots  (\ref{eq4.14})  depends on the finite index open 
subgroups of the  $\mathbf{Q}_p^{\times}$.

\bigskip
(iii) Let us restrict to the case $\mathscr{W}^4\cong S^4$ is a 4-dimensional sphere. 
By theorem \ref{thm4.1} (i), we have $F(S^4)=\mathbb{H}(\mathbf{Q})$ and from (\ref{eq4.15}):
%***************************************************************************************
\begin{equation}\label{eq4.17}
\mathscr{D}_p(\mathbb{H}(\mathbf{Q}))\cong \mathbb{H}(\mathbf{Q}_p). 
\end{equation}
%******************************************************************************
Since the index $n=1$, we conclude that the sphere knots:
%***************************************************************************************
\begin{equation}\label{eq4.18}
S^2\hookrightarrow S^4 
\end{equation}
%******************************************************************************
are classified by the primes $p$.  In view of theorem \ref{thm2.5} (iii) and 
that $\mathscr{P}\cong p\mathbf{Z}$, we recover 
 the Artin  classification of sphere knots 
$\pi_1(S^4-S^2)\cong\pi_1(S^3-\mathscr{K})$
by  knots  in $S^3$  [Kamada 2002] \cite[Chapter 10]{K}.
\end{example}
%****************************************************************************

\bigskip
%***************************************************************************
\begin{remark}
It is known, that all division algebras over a local field are cyclic. 
The Brauer group of such algebras has the form:
%***************************************************************************************
\begin{equation}\label{eq4.19}
Br~(\mathscr{C}(a, K_p^{ab}))\cong \mathbf{Q}/\mathbf{Z},
\end{equation}
%******************************************************************************
see e.g. \cite[Section 17.10]{P}.  The Brauer group classifies the 
Morita equivalence classes of the algebras  $\mathscr{C}(a, K_p^{ab})$ and,
therefore, the corresponding classes of the sphere knots (\ref{eq4.14}). 
However, the nature of such  classes  is unclear
to the author. 
\end{remark}
%**************************************************************************

\medskip
%***************************************************************************
\begin{remark}\label{rmk4.7}
Using (\ref{eq4.15}),  one rewrite (\ref{eq4.19}) in the form:
%***************************************************************************************
\begin{equation}\label{eq4.20}
Br ~(\mathscr{D}_p)\cong \mathbf{Q}/\mathbf{Z}.
\end{equation}
%******************************************************************************
Notice a formal resemblance of (\ref{eq4.20}) to (\ref{eq1.2}). 
This correspondence is part of the Merkurjev-Suslin theory linking the Brauer groups with  the algebraic K-theory.  
\end{remark}
%**************************************************************************
\bibliographystyle{amsplain}

\begin{thebibliography}{99}

\bibitem{Ded1}
R.~Dedekind,
\textit{\"Uber die Theorie der ganzen algebraischen Zahlen}, 
in Vorlesungen \"uber Zahlentheorie, 1863. 


\bibitem{GH}
W.~H.~Gottschalk and G.~A.~Hedlund, 
\textit{Topological Dynamics}, AMS Colloquium Publications
{\bf 36}, 1955.   


\bibitem{K}
S.~Kamada,
\textit{Braid and Knot Theory in Dimension Four},
Mathematical Surveys and Monographs {\bf 95}, 
AMS, 2002.  


\bibitem{Kum1}
E.~E.~Kummer,
\textit{\"Uber die Zerlegung der aus Wurzeln der Einheit gebildeten complexen Zahlen in ihre Primfactoren},
Journal f\"ur die reine und angewandte Mathematik {\bf 35} (1847), 327-367.

\bibitem{M}
M. ~Morishita, \textit{Knots and Primes, 
An Introduction to Arithmetic Topology}, 
Springer Universitext,  London, Dordrecht, Heidelberg, New York, 
2012. 


\bibitem{N}
I.~V.~Nikolaev, \textit{Noncommutative Geometry},
De Gruyter Studies in Math. {\bf 66}, Berlin, 2017.


\bibitem{Nik1}
I.~V.~Nikolaev, 
\textit{Remark on arithmetic topology}, arXiv:1706.06398


\bibitem{Nik2}
I. ~V.~Nikolaev, 
\textit{Arithmetic topology of 4-manifolds}, 	arXiv:1907.03901


\bibitem{Nik3}
I. ~V.~Nikolaev, 
\textit{Untying knots in 4D and Wedderburn's Theorem}, 
Comm. Algebra {\bf 49} (2021), 4133-4135. 


\bibitem{P}
R.~S.~Pierce, 
\textit{Associative Algebras}, GTM {\bf 88}, Springer, 1982.   

\bibitem{Pie1}
R.~Piergallini, 
\textit{Four-manifolds as 4-fold branched covers of $S^4$}, 
Topology {\bf 34} (1995), 497-508. 

\bibitem{Wei1}
A.~Weil,  \textit{Numbers of solutions of equations in finite fields},   Bull. Amer. Math. Soc. {\bf 55}
(1949), 497-508.  


\end{thebibliography}

%**********************************************************

\end{document}